\documentclass[11pt,leqno]{amsart}

\usepackage{amssymb,hyperref}
\usepackage{longtable}

\newtheorem{theorem}{Theorem} 
\newtheorem{lemma}[theorem]{Lemma}
\newtheorem{proposition}[theorem]{Proposition}

\theoremstyle{definition}

\newtheorem{remark}[theorem]{Remark}

\numberwithin{equation}{section}

\def\bb{\mathbb}
\def\bf{\mathbf}
\def\cal{\mathcal}
\def\frak{\mathfrak} 

\def\ZZ{\bb{Z}}
\def\NN{\bb{N}}

\def\NN{\bb{N}}
\def\ZZ{\bb{Z}}

\def\rm{\textrm}

\def\r{\rangle}
\def\a{\alpha}

\def\bb{\mathbb}
\def\cal{\mathcal}
\def\rm{\textrm}
\def\bf{\textbf}

\def\frak{\mathfrak}

\def\Frat{\Phi}
\def\soc{\mathrm{soc}}
\def\aut{\mathrm{Aut}}

\def\PSL{\mathrm{PSL}}
\def\Alt{\mathrm{Alt}}

\def\Sym{\mathrm{Sym}}

\title[Rational probabilistic zeta function]{Finiteness of
profinite groups \\ with a rational probabilistic zeta function}

\author[Duong Hoang Dung]{Duong Hoang Dung} 
\address{Duong Hoang Dung: Fakult\"{a}t f\"{u}r Mathematik,
Universit\"{a}t Bielefeld,
Postfach 100131,
D-33501 Bielefeld, Germany.}
\email{dhoang@math.uni-bielefeld.de}



\begin{document}

\begin{abstract}
We prove that every  profinite group in a certain class with 
a rational probabilistic zeta function has only finitely
many maximal subgroups.
\end{abstract}

\maketitle

\section{Introduction}
Let $G$ be a finitely generated profinite group. Given an
integer $n\in\NN$, there are only finitely many index $n$ subgroups of $G$; all of them are open (cf. \cite{Nikolov}).
We define the integer $a_n$ to be the finite sum
$\sum_{|G:H|=n}\mu_G(H)$ where $\mu_G$ is the M\"{o}bius function defined
on the lattice of open subgroups of $G$ recursively by $\mu_G(G)=1$ and $\sum_{H\leq K\leq G}\mu_G(K)=0$ if $H<G$.
A fruitful approach tool to study the sequence $\{a_n\}_{n\in\NN}$ is its associated Dirichlet series
$$P_G(s)=\sum_{n=1}^\infty \frac{a_n}{n^s},$$
where $s$ is a complex variable. For example, if $G=\widehat{\ZZ}$ is the profinite completion of the integers
$\ZZ$, then it is easy to see that $P_{\widehat{\ZZ}}(s)=
\zeta(s)^{-1}$, the inverse of the Riemann zeta function.

If $G$ is a finite group, then for every positive integer $t$, the value $P_G(t)$ is the probability that $t$ randomly chosen elements generate the group $G$ (cf. \cite{Hall}). In
the context of profinite groups, Mann conjectured in 
{\cite{Mann96}} that if $G$ is  positively finitely generated, i.e., the probability $\mathrm{Prob}_G(t)$ that $t$ randomly chosen elements generate $G$ topologically is positive for some positive integer $t$, then $P_G(s)$ converges in some right half-plane and $P_G(t)=\mathrm{Prob}_G(t)$ for $t\in\mathbb{N}$ large enough. The conjecture has been verified for several classes
of profinite groups (cf. \cite{Mann05},\cite{Lu07},\cite{Lu11b}) and the study of
convergency of $P_G(s)$ is reduced to the study of subgroup
structures of finite almost simple groups (cf. \cite{Lu10}, \cite{Lu11a}) which has been done
for symmetric and alternating groups (cf. \cite{CL10}). Despite the convergence of $P_G(s)$, the reciprocal of $P_G(s)$ is usually called the
\emph{probabilistic} zeta function of $G$ (cf. \cite{Mann96},
\cite{Boston}).

Notice that if $\mu_G(H)\ne 0$ then $H$ is an intersection
of maximal subgroup (cf. \cite{Hall}), and thus $H$ contains
the Frattini subgroup $\Frat(G)$ of $G$, which is the intersection
of the maximal open subgroups of $G$.
Hence the series
$P_G(s)=P_{G/\Frat(G)}(s)$ encodes information about the lattice generated
by maximal subgroups of $G$, just as the Riemann zeta function encodes information
about the primes.
It is obvious that if $\Frat(G)$ is open, i.e., $G$ contains only finitely
many maximal subgroups, then there are only finitely many open subgroups $H$
with non-zero M\"obius values. Therefore, $a_n=0$ for all but finitely many $n\in\NN$,
and the series $P_G(s)$ is a finite series. It is natural to ask whether the
converse is still true. It is conjectured in \cite{DeLu06} that for a finitely
generated profinite group $G$, if the probabilistic zeta function is rational, i.e.,
a quotient of two finite Dirichlet series, then $G/\Frat(G)$ is finite. The conjecture
have been confirmed to be true for prosolvable groups (cf. \cite{DeLu06}) and several
classes of non-prosolvable groups (cf. \cite{DeLu07}, \cite{Dung-Lie}, \cite{Dung-PSL}).
A consequence of those results is that if almost every composition factor of $G$ is 
isomorphic to a given simple group $S$, then the conjecture holds. The problem becomes
quite difficult when one replaces $S$ by a finite set $\frak{S}$ of finite simple groups,
even in the case $\frak{S}$ is a family of comparable finite simple groups, e.g., finite
simple groups of Lie type over arbitrary characteristics.
Fortunately, the problem can be solved in \cite{Dung-PSL} if every simple group in $\frak{S}$ is $\PSL(2,p)$ for some prime $p\geq 5$. However, the conjecture is still open in general and it is very hard to give an affirmative answer for the conjecture when $\frak{S}$ contains non-comparable families of finite simple groups. In this paper, using the methods developed in \cite{Dung-PSL} and \cite{Dung-Lie}, we obtain the following result.

\begin{theorem}\label{main theorem}
Let $G$ be a finitely generated profinite group and $H$ an
open subgroup of $G$ such that every nonabelian composition
factor of $H$ is isomorphic to one of the following:
\begin{itemize}
\item $\PSL(2,p)$ for some prime $p\geq 5$;
\item a sporadic simple group;
\item  $\Alt(n)$ where $n$ is either a prime or a $2$-power.
\end{itemize}
Then $P_G(s)$ is rational only if $G$ contains only finitely
many maximal subgroups.
\end{theorem}

\section{Preliminaries}
Let $G$ be a finitely generated profinite group and $\{G_i\}_{i\in\bb{N}}$  a fixed countable descending series of open normal subgroups with the property that $G_0=G$, $\cap_{i\in\bb{N}}G_i=1$ and $G_i/G_{i+1}$ is a chief factor of $G/G_{i+1}$ for each $i\in\bb{N}$. In particular, for
each $i \in \bb{N}$, there exist a simple group $S_i$ (composition factor) and a positive integer $r_i$ (composition length) such that $G_i/G_{i+1}\cong S_i^{r_i}$.
It is shown in \cite{DeLu06b} that $P_G(s)$ can be written as an infinite formal product
\begin{equation*}
P_G(s)=\prod_{i\in \bb{N}}P_i(s).
\end{equation*}
where for each $P_i(s)=\sum_{n\in\NN}b_{i,n}n^{-s}$ is the finite Dirichlet series 
associated to the chief factor $G_i/G_{i+1}$. Moreover, this factorization is independent on the choice of chief series (see \cite{DeLu03,DeLu06b}) and $P_i(s)=1$ unless $G_i/G_{i+1}$ is a non-Frattini chief factor of $G$, i.e., $G_i/G_{i+1}$ is
not contained in the Frattini subgroup $\Frat(G/G_{i+1})$.

If $S_i$ is cyclic of order $p_i$ then  $P_i(s)=1-c_i/(p_i^{r_i})^s $ where $c_i$ is the number
of complements of $G_i/G_{i+1}$ in $G/G_{i+1}$ (cf. \cite{Gaschutz}). When $S_i$ is nonabelian,
the series $P_i(s)$ is complicated. If $n\ne |S_i|^{r_i}$ then 
\begin{equation*}
b_{i,n}=\sum_{\substack{|L_i:H|=n\\L_i=H \soc(L_i)}}\mu_{L_i}(H),
\end{equation*}
where $L_i:=G/C_G(G_i/G_{i+1})$ is the monolithic primitive group with unique minimal normal subgroup
$\soc(L_i)=G_i/G_{i+1}$ (cf. \cite{DeLu03}); here $\soc(L_i)$ is the socle of $L_i$. Computing these coefficients $b_{i,n}$ is reduced to studying the subgroup structures
of the group $X_i\leq \aut S_i$ induced by the conjugation action of
the normalizer in $L_i$ of a composition factor of the socle $S_i^{r_i}$; note that $X_i$ is an almost simple group with socle isomorphic to $S_i$. More precisely, given an almost simple group $X$
with socle $S$, define the following finite Dirichlet series
$$P_{X,S}(s)=\sum_{n}\frac{c_n(X)}{n^s},~\textrm{\ where $c_n(X)=\sum_{\substack{|X:H|=n\\X=SH}}\mu_X(H)$.}$$
The following useful Lemma which is deduced from {\cite{Seral}}.
\begin{lemma}\label{paz}
If $S_i$ is nonabelian and $\pi$ is the set of primes containing at
least one divisor of $|S_i|$, then
$$P_i^{\pi}(s)= P^{\pi}_{X_i,S_i}(r_is-r_i+1).$$
In particular, if $n$ is not divisible by some prime in $\pi$, then there exists $m\in \Bbb N$ with $n=m^{r_i}$ and $b_{i,n}=c_m(X_i)\cdot m^{r_i-1}.$
\end{lemma}
Here for a Dirichlet series $F(s)=\sum_na_nn^{-s}$ and a given set of primes $\pi$, the 
series $F^\pi(s)$ is obtained from $F(s)$ by deleting  terms $a_nn^{-s}$ whenever $n$ is divisible by some prime in $\pi$. 
Notice that if $F(s)$ is rational then so is $F^\pi(s)$.

For an almost simple group $X$, let $\Omega(X)$ be the set of the odd integers
$m\in \Bbb N$ such that
\begin{itemize}
\item $X$ contains at least one subgroup $Y$ such that $X=Y\soc X$ and $|X:Y|=m;$
\item if $X=Y\soc X $ and $|X:Y|=m,$ then $Y$ is a maximal subgroup if $X.$
\end{itemize}
Note that if $m \in \Omega(X),$ $X=Y\soc X $ and $|X:Y|=m,$
then $\mu_X(Y)=-1$: in particular
$c_m(X)<0.$ Combined with Lemma \ref{paz}, this implies:
\begin{remark}
If $m\in \Omega(X_i)$ then $b_{i,m^{r_i}}<0$. 
We say such an $m$ a \emph{useful index} for $L_i$.
\end{remark}

We are now looking for $w(X):=\min\{m|m\in\Omega(X)\}$ where $X$ is an almost simple group whose socle $S$ a finite simple group appeared in Theorem \ref{main theorem}. In case $S$ is either $\PSL(2,p)$ for some prime $p\geq 5$ or a sporadic simple group, the values $w(X)$ are
computed in \cite{Dung-PSL} and \cite{Dung-Lie} respectively. When
$S=\Alt(p)$ for a prime $p\geq 5$, then $\Alt(p-1)$ is the unique maximal  subgroup of $S$ of index $p$ (cf. \cite{Gur83}) hence $w(X)=p$.
It remains to compute $w(X)$ for $S=\Alt(n)$ when $n=2^t$ is a
$2$-power with $t\geq 3$.

Now let $X$ be an almost simple group with socle $S=\Alt(n)$ when $n=2^t$ is a $2$-power with $t\geq 3$. Let $M$ be a maximal subgroup of $X$ such that $M$ contains a Sylow $2$-subgroup of $X$. By Liebeck and Saxl (\cite{LiSa85}), the group $M$ is of either intransitive type or imprimitive type unless $n=8$ and $X=\Alt(8)$. 

In case $M$ is of intransitive type, i.e., $M=(\Sym(k)\times \Sym(n-k))\cap X$ with $1\leq k<n/2$, then 
$$|X:M| =\binom{n}{k}.$$
Let $n=\sum_{i}n_i2^i$ and $k=\sum_{i}k_i2^i$ be $2$-adic expansions of $n$ and $k$ respectively, then applying Lucas' Theorem gives us
$$|X:M|=\binom{n}{k}\equiv\prod_i\binom{n_i}{k_i}\equiv 0\mod 2$$
since $n_i=0$ for all but $i=t$. Therefore $M$ is always of even index in $X$, a contradiction. Hence $M$ is not intransitive.

In case $M$ is of imprimitive type, i.e., $M=(\Sym(a)\wr \Sym(b)\cap X$ where $ab=2^t, a>1, b>1$, then $M$ has index
$$w(a,b)=\frac{(2^t)!}{(a!)^b.b!}.$$
Notice that $w(a,b)$ is always odd because of the following. For a given integer $n$, let $n=a_0+a_1p+\cdots +a_dp^d$ be the $p$-adic expansion of $n$ where $p$ is a prime and $0\leq a_j\leq p-1$ for each $j$, $a_d\ne 0$. Let $S_p(n)=a_0+\cdots+a_d$, then Legendre's formula gives us
$$v_p(n!)=\frac{n-S_p(n)}{p-1}.$$ 
Now in this case $n=2^t$, then $p=2, a_t=1$ and $a_k=0$ for all $k\ne t$. So $S_2(n)=1$. We have that $v_2(n!)=2^t-1$. Assume that $a=2^u, b=2^v$ then $v_2(a)=2^u-1$, $v_2(b)=2^v-1$ so that
$$v_2((2^u!)^{2^v})=(2^u-1)2^v.$$
Then
$$v_2((2^u!)^{2^v}.2^v!)=(2^u-1)2^v+2^v-1=2^{u+v}-1=2^t-1.$$
Hence $w(a,b)$ is always odd.

In case $n=8$ and $X=\Alt(8)$ then $M=2^3:\rm{PSL}(3,2)$, and so $|X:M|=3\cdot5$.

\begin{lemma}{\cite[Lemma 2.1]{Maroti02}}
Assume that $m=ab\geq 8, a,b\geq 2$. Then
$$w(a,b)=\frac{m!}{a!^b.b!}$$
is smallest when $b$ is the smallest prime divisor of $n$.
\end{lemma}

In particular, when $n=2^t$, then $w(2^{t-1},2)$ is the smallest among $w(a,b)$ with $n=ab$, $a>1,b>1$. Hence, we obtain the following.
\begin{proposition}\label{smallest useful index chapter 6}
Let $X$ be an almost simple group with socle $\Alt(2^t)$. Then we have:
$$w(X)=\left\{
\begin{array}{cc}
(2^t)!/((2^{t-1}!)^2\cdot 2!) & \rm{if $t>3$}, \\
3\cdot 5 & \rm{if $t=3$}.
\end{array}
\right.$$
\end{proposition}

\section{Proof of Theorem \ref{main theorem}}
We start now the proof of our main result. Let $G$ now be a finitely generated profinite group of which almost every nonabelian composition factor is isomorphic to either a sporadic group, or a projective special linear group $\rm{PSL}(2,p)$ for some prime $p\geq 5$, or an alternating group $\Alt(n)$ where $n$ is either a prime or a power of $2$. Assume in addition that $P_G(s)$ is rational.

We fix a descending normal series $\{G_i\}$ of $G$ with the properties that $\bigcap G_i=1$ and $G_i/G_{i+1}$ is a chief factor of $G/G_{i+1}$. Let $J$ be the set of indices $i$ with $G_i/G_{i+1}$ non-Frattini. Then  the probabilistic zeta function $P_G(s)$ can be factorized as 
$$P_G(s)=\prod_{i\in J}P_i(s)$$
where for each $i$, the finite Dirichlet series $P_i(s)$ is associated to the chief factor $G_i/G_{i+1}$.

For $C(s)=\sum_{n=1}^\infty  c_n/n^s\in \cal R$, we define $\pi(C(s))$ to be the set of the primes $q$ for which there exists
at least one multiple $n$ of $q$ with $c_n\ne 0$. Notice that if $C(s)=A(s)/B(s)$ is rational then $\pi(C(s))\subseteq \pi(A(s))\bigcup \pi(B(s))$ is finite.

Let $\Gamma$ be the set of the finite simple groups that are isomorphic to a composition factor of some non-Frattini chief factor of $G$. The first step in the proof of Theorem \ref{main theorem} is to show that $\Gamma$ is finite. The proof of this claim requires the following result.
\begin{lemma}[{\cite[Lemma 3.1]{DeLu07}}]\label{number of non-iso composition factors}
Let $G$ be a finitely generated profinite group and let $q$ be a prime with $q\notin \pi(P_G(s))$.
If $q$ divides the order of a non-Frattini chief factor of $G,$ then this factor is not a $q$-group.
\end{lemma}

Let $\pi(G)$ be the set of the primes $q$ with the properties that $G$ contains at least an open subgroup $H$ whose index is divisible by $q$.
\begin{proposition}
The set $\pi(G)$ is finite.
\end{proposition}
\begin{proof}
Let $\Gamma^*$ be the subset of $\Gamma$ containing simple groups $S_i\in\Gamma$ such that $S_i$ is isomorphic to either a sporadic simple group, or $\rm{PSL}(2,p)$ for some prime $p\geq 5$, or an alternating group $\rm{Alt}(n)$ where $n$ is either a prime or a power of $2$.
Since $P_G(s)$ is rational, the set $\pi(P_G(s))$ is a finite set. It follows from Lemma \ref{number of non-iso composition factors} that $\Gamma$ contains only finitely many abelian groups, i.e.,  $F\setminus F^*$ is finite.

Notice that for each simple group $S\in\Gamma^*$ which is not sporadic, either there is an odd prime $r$ such that $S\cong\rm{PSL}(2,r)$ or $S\cong \rm{Alt}(r)$ or $S\cong \rm{Alt}(2^t)$ for some $t\geq 3$. In the latter case, there is also an associated prime $p_t$, namely the largest prime less than $2^t$. Notice also that for a given prime $p$, there are only finitely many numbers $n$ such that $p$ is the largest prime less than $n$, since for every large enough positive integer $x$, there is always a prime in the interval $[x,6/5x]$ (see \cite{Nagura}). 

Assume by contradiction that $\Gamma^*$ is infinite. Let $I$ be the set of indices $i$ such that $S_i\in\Gamma^*$ and let $A(s)=\prod_{i\in I}P_i(s)$, $B(s)=\prod_{i\in J\setminus I}P_i(s)$. Since $\pi(P_G(s))$ is finite and by hypothesis, $\pi(B(s))$ is finite, it follows that $\pi(A(s))$ is finite. In particular, there is a prime $p>71$ such that $p\notin\pi(A(s))$. We choose the prime number $p>71$ since $71$ is the largest prime dividing the order of a sporadic simple group, so if $p$ divides some $|S_i|$ then $S_i$ is either an alternating group or a projective special linear group.  Set
$$I^*=\{i\in I : p\in\pi(P_i(s))\}\quad\rm{and}\quad r=\min\{r_i : i\in I^*\}.$$

Let $\Lambda$ be the set of all odd positive integers $n$ such that $n$ is divisible by $p$ and is not divisible by any prime strictly larger than $p$. Notice that for each $i\in I^*$, there exists $\alpha\in\Lambda$ such that $\alpha$ is a useful index for $L_i$. Indeed, if $p\in\pi(P_i(s))$ then one of the following occurs:
\begin{itemize}
\item $S_i=\rm{PSL}(2,p)$ and $(p(p\pm 1)/2)^{r_i}$ are useful indices for $L_i$.
\item $S_i=\rm{Alt}(p)$ and $p^{r_i}$ is a useful index for $L_i$.
\item $S_i=\rm{Alt}(2^t)$ and $w(2^t,2)$ is a useful index for $L_i$.
\end{itemize}

Choose $\a\in\Lambda$ minimal with the properties that there exists $i\in I^*$ such that $r_i=r$, and $v_p(\a)=r$ and $\a$ is a useful index for $L_i$. Notice that if $i\in I$ such that $b_{i,\a}\ne 0$ then $i\in I^*$, and $r_i=r$ and $\a$ is the smallest odd useful index for $L_i$ divisible by $p$. It follows that $b_{i,\a}<0$. Hence the coefficient $c_\a$ of $1/\a^s$ in $A(s)$ is

$$c_\a=\sum_{\substack{i\in I\\r=r_i}}b_{i,\a}=\sum_{i\in I^*}b_{i,\a}<0$$
which implies that $p\in\pi(A(s))$, a contradiction. Thus $\Gamma^*$, and hence $\Gamma$ is finite. Consequently, the set $\pi(G)$ is also finite.
\end{proof}
For a simple group $S\in\Gamma$, let $I_S=\{j\in J \mid S_j \cong S\}$.
Our aim is to prove that, under the hypotheses of Theorem \ref{main theorem},  $J$ is a finite set. We have already proved that $\Gamma$ is finite, so it suffices to prove that $I_S$ is finite for each $S \in \Gamma.$ If $S$ is abelian, then $I_S$ is a finite set
(cf. {\cite[Lemma 4.4]{Dung-Lie}}). It remains to consider the nonabelian cases.

The following useful consequence of Skolem-Mahler-Lech theorem is an important tool in deducing finiteness of
certain infinite informal products of finite Dirichlet series.

\begin{proposition}{\cite{DeLu06}}\label{skolem remark}
Let $G$ be a finitely generated profinite group with $\pi(G)$ finite and $\{r_i\}$ the sequence of the composition lengths
of the non-Frattini factors in a chief series of $G$. Assume that there exists a positive integer $q$ and a sequence ${c_i}$ of nonnegative integers
such that the formal product
$$\prod_{i}\left(1-\frac{c_i}{(q^{r_i})^s}\right)$$
is rational. Then $c_i=0$ for all but finitely many indices $i.$
\end{proposition}
\begin{proof}[\bf{Proof of Theorem \ref{main theorem}}]
Let $\cal{T}$ be the set of almost simple groups $X$ such that there are infinitely many indices $i\in J$ with $X_i\cong X$, and let $I=\{i\in J : X_i\cong X\}$. Our assumptions combined with {\cite[Lemma 4.4]{Dung-Lie}}, imply that $J\setminus I$ is finite. In order to prove that $J$ is finite, we need to prove that $I$ is empty. Since $P_G(s)$ is rational, also 
$\prod_{i\in I}P_i(s)$ is rational.
Thus the product $\prod_{i\in I}P_i^{(2)}(s)$ is still rational. Notice that the set $\pi(\cal{T})$ of prime divisors of orders of simple groups $S$ such that $S=\rm{soc}(X)$ for some $X\in\cal{T}$ is finite. Let
$$q=\max\{p : p\in\pi(\cal{T})\}$$
and $$\Lambda_q=\{i\in I : q\in\pi(P_i^{(2)}(s))\}.$$
Notice that if $q>71$ and $i\in \Lambda_q$ then $S_i$ is isomorphic to either $\rm{Alt}(q)$ or $\rm{Alt}(2^t)$, where $q$ is the largest prime less than $2^t$, or $\rm{PSL}(2,q)$. Let
\begin{eqnarray*}
w & = &  \min \{ x\in\bb{N} \mid x~\rm{is odd}, v_{q}(x)=1~\rm{and}~b_{i,x^{r_i}}\ne 0~\rm{for some $i$} \} \\
 & = & \min \{ x\in\bb{N} \mid x~\rm{is odd}, v_{q}(x)=1~\rm{and}~b_{i,x^{r_i}}\ne 0~\rm{for some $i\in \Lambda_{q}$} \}.
\end{eqnarray*}
Note that if $i\in\Lambda_q$, and $n$ is minimal with the properties that $n$ is odd, $b_{i,n^{r_i}}\ne 0$ and $v_q(n)=1$, then $b_{i,n^{r_i}}<0$. So if $b_{i,n^{r_i}}\ne 0$ then $b_{i,n^{r_i}}<0$. Moreover, by {\cite[Proposition 2.3]{Dung-Lie}}, we obtain the rational product
$$\prod_{i\in\Lambda_{q}}\left(1+\frac{b_{i,w^{r_i}}}{(w^{r_i})^s}\right)$$
where $b_{i,w^{r_i}}\leq 0$ for each $i\in\Lambda_q$. By applying Proposition \ref{skolem remark}, we get that $\Lambda^*_{q}=\{i\in\Lambda_q\mid b_{i,w^{r_i}}\ne 0 \}$ is finite, but this implies that $\Lambda_q=\emptyset$. Let 
$$p=\max\{r : r\in\pi(\cal{T})\setminus\{q\} \}$$
and $$\Lambda_p=\{i\in I : p\in\pi(P_i^{(2)}(s))\}.$$
Again, by the same argument, we obtain that $\Lambda_p=\emptyset$. 
We continue this procedure until we reach the prime $71$. Notice that we deal with both $\rm{PSL}(2,71), \rm{Alt}(71)$ and $\rm{M}$ by the prime $71$. Now, with the same spirit, we deal with all remaining sporadic groups, alternating groups and $\rm{PSL}(2,q)$ with $5\leq q<71$. Our process is described in order in the following table in which the first column stands for the primes $r$ we use in order to deduce that $\Lambda_{r}=\emptyset$, the second column are simple groups $S$ involved corresponding to the primes in the first column:
\begin{longtable}{|l|l|}
\hline
$r$ & $S$  \\
\hline 
$71$ & $\PSL(2,71), \mathrm{M}, \Alt(71)$ \\
\hline 
$67$ & $\PSL(2,67),\Alt(67), \mathrm{Ly}$ \\
\hline 
$61$ & $\Alt(61), \Alt(2^6)$ \\
\hline 
$59$ & $\PSL(2,59), \Alt(59)$ \\
\hline 
$53$ &  $\PSL(2,53), \Alt(53)$\\
\hline 
$47$ & $\PSL(2,47), \Alt(47), \mathrm{BM}$ \\
\hline 
$43$ & $\PSL(2,43), \Alt(43), \mathrm{J}_4$\\
\hline 
$41$ &  $\PSL(2,41), \Alt(41)$\\
\hline 
$37$ & $\PSL(2,37), \Alt(37)$\\
\hline 
$31$ & $\PSL(2,31), \Alt(31), \Alt(2^5), \mathrm{O'N}$\\
\hline 
$29$ & $\PSL(2,29), \Alt(29), \mathrm{Fi'}_{24}, \mathrm{Ru}$ \\
\hline 
$23$ & $\PSL(2,23), \Alt(23), \mathrm{M}_{23}, \mathrm{M}_{24}, \mathrm{Co}_2, \mathrm{Co}_3, \mathrm{Fi}_{23}$\\
\hline 
$19$ &  $\PSL(2,19), \Alt(19), \mathrm{J}_1, \mathrm{J}_3, \mathrm{HN}, \mathrm{Th}$\\
\hline 
$17$ &  $\PSL(2,17), \Alt(17), \mathrm{He}$\\
\hline 
$13$ & $\PSL(2,13), \Alt(13), \Alt(2^4), \mathrm{Suz}, \mathrm{Co}_1, \mathrm{Fi}_{22}$\\
\hline 
$11$ &  $\PSL(2,11), \Alt(11), \mathrm{M}_{11}, \mathrm{M}_{12}, \mathrm{M}_{22}, \mathrm{HS}, \mathrm{McL}$\\
\hline 
$7$ &  $\PSL(2,7), \Alt(7), \mathrm{J}_2$\\
\hline 
$5$ &$\PSL(2,5), \Alt(5), \Alt(2^3)$\\
\hline 
\end{longtable}
\end{proof}

\end{document}